\documentclass[11pt]{amsproc}
\usepackage{amssymb, amsmath,amsthm,color,enumerate,booktabs,upgreek,tikz}
\usepackage{url,hyperref}
\hypersetup{citecolor=blue, linkcolor=blue, colorlinks=true}
\usepackage[marginparwidth=1.5cm,margin=2cm]{geometry}

\theoremstyle{plain}
\newtheorem{theorem}{Theorem}[section]
\newtheorem{lemma}[theorem]{Lemma}
\newtheorem{result}[theorem]{Result}
\newtheorem{remark}[theorem]{Remark}
\newtheorem{example}[theorem]{Example}

\newcommand\PG{\mathsf{PG}}
\renewcommand\le{\leqslant}
\renewcommand\ge{\geqslant}

\newcommand{\erz}[1]{\langle #1\rangle}

\newcommand\F{\mathcal{F}}

\newcommand\calP{\mathcal{P}}
\newcommand\calQ{\mathcal{Q}}
\newcommand\calE{\mathcal{E}}
\newcommand\calH{\mathcal{H}}

\newcommand\calPdash{X}

\title{On intriguing sets of the Penttila-Williford association scheme}

\author{John Bamberg}
\address[Bamberg]{Centre for the Mathematics of Symmetry and Computation, Department of Mathematics and Statistics, The University of Western Australia, 35 Stirling Highway, Crawley, W. A. 6019, Australia.}
\author{Klaus Metsch}
\address[Metsch]{Justus-Liebig-Universit\"{a}t, Mathematisches Institut,
Arndtstra{\ss}e 2, D-35392 Gie{\ss}en, Germany.}

\date{}

\keywords{elliptic quadric, intriguing set, association scheme}

\subjclass[2000]{05E30, 51E12}

\begin{document}

\maketitle

\begin{abstract}
We investigate intriguing sets of an association scheme introduced by Penttila and Williford (2011) that was the basis for
their construction of primitive cometric association schemes that are not $P$-polynomial nor the dual of a $P$-polynomial scheme. In particular, we give examples and characterisation results for the four types of intriguing sets that arise in this scheme.
\end{abstract}

\section{Introduction}

A celebrated result of Beniamino Segre \cite[n. 91]{Segre:1965aa} is that if there exists a nonempty proper subset $S$ of the point set $\calP$ of the elliptic quadric $\mathsf{Q}^-(5,q)$, $q$ odd, such that every generator meets $S$ in a constant number $m$ of points, then $m=(q+1)/2$. Such a set $S$ is called a \emph{hemisystem}, and by the work of Cossidente and Penttila \cite{Cossidente:2005aa}, these configurations exist for every odd prime power $q$. Hemisystems garnered interest from the algebraic combinatorics community, as they give rise to cometric $Q$-antipodal association schemes \cite{Dam:2013aa}. Segre's result was extended to $q$ even by Bruen and Hirschfeld \cite{Bruen:1978aa}, who showed that no such subsets $S$ can exist. Thus for $q$ even, there is an absence of such interesting configurations. However, Penttila and Williford \cite{Penttila:2011aa} introduced the notion of a \emph{relative hemisystem} of $\mathsf{Q}^-(5,q)$, $q>2$ even, which give rise to atypical and rare association schemes; primitive cometric association schemes that do not arise from distance regular graphs. Their idea was to consider a non-tangent hyperplane $H$ for $\mathsf{Q}^-(5,q)$ and to only regard the points $\calPdash$ that lie outside of $H$. Now, a nonempty proper subset $S$ of $\calPdash$ such that every generator meets $S$ in a constant number $m$ of points must have $m=q/2$, according to \cite[Theorem 1]{Bamberg:2017aa}. In particular, $q$ is even. Such a configuration is a \emph{relative hemisystem}.

In the background, there is a $4$-class cometric association scheme, which we call the \emph{Penttila-Williford} scheme. It arises from taking the natural relations that are invariant under the stabiliser of $H$ in the full similarity group of $\mathsf{Q}^-(5,q)$, and it exists for all prime powers $q$, odd or even (but greater than $2$). Table \ref{tbl:summary} summarises the relations of this scheme, and full details will be given in Section \ref{section:PW}.

\begin{table}[ht]
\begin{center}
\begin{tabular}{cl}
\toprule
Relations& Description\\
\midrule
$R_0$& equality\\
$R_1$& noncollinear but collinear to the conjugate of the other\\
$R_2$& noncollinear and noncollinear to the conjugate of the other\\
$R_3$& collinear and not conjugate\\
$R_4$& conjugate \\
\bottomrule
\end{tabular}
\end{center}
\caption{The Penttila-Williford Scheme}\label{tbl:summary}
\end{table}

A relative hemisystem, if it exists, provides an example of an \emph{intriguing set} for this association scheme, in the language of De Bruyn and Suzuki \cite{De-Bruyn:2010aa}. Little is known about other intriguing sets for this association scheme, apart from the devices used in Melissa Lee's MPhil thesis \cite{Lee:2016aa} and in \cite{Bamberg:2017aa}. In this paper, we begin an investigation into the various intriguing sets for this scheme and we derive characterisation results.
Since there are four non-principal eigenspaces, there are four `types' $i\in\{1,2,3,4\}$ of intriguing sets depending on the index $i$ for which its characteristic vector belongs to the sum $V_0\oplus V_i$ of eigenspaces for the Penttila-Williford scheme. There is a particular involution $\sigma$ that acts fixed-point-freely on $\calPdash$, and we say that two points $x,y\in \calPdash$ are \emph{conjugate} if $x=y^\sigma$.
A subset $S$ of $\calPdash$ is \emph{$\sigma$-invariant}, if $x^\sigma\in \calPdash$ for all $x\in \calPdash$.
If an intriguing set $S$ is not $\sigma$-invariant, then the image $S^\sigma$ of $S$ under $\sigma$ is disjoint $S$, and we shall see that $|S|=\tfrac{1}{2}|\calPdash|$.

\begin{theorem}[Paraphrase of Theorem \ref{hemi}]\label{hemi_paraphrased}
Suppose that $Y$ is an intriguing set of type $i$, where $i\in\{1,2,3,4\}$.
\begin{enumerate}[(a)]
\item If $i=1$ or $i=3$, then $|Y|=\frac12|X|$ and $|\{p,p^\sigma\}\cap Y|=1$ for all $p\in X$.
\item If $i=2$ or $i=4$, then $Y^\sigma=Y$.
\end{enumerate}
\end{theorem}

We show in Section \ref{section:tight} that an intriguing set of type 2 has size at least $4(q+1)$ (see Lemma \ref{alpha4}), and this bound is sharp.
In Theorem \ref{characterise_type2}, we characterise the smallest examples for when $q\ge 59$.
We also show in Section \ref{type4} that an intriguing set of type 4 has size at least $q^2(q-1)$, and this bound is sharp.
The following result characterises the smallest examples.

\begin{theorem}[Paraphrase of Theorem \ref{smalltype4}]\label{characterisation_type4}
Suppose $Y$ is an intriguing set of type 4 of the Penttila-Williford scheme. Then $|Y|\ge q^2(q-1)$, and in the case of equality, there exists a point $p\in H\cap \mathsf{Q}^-(5,q)$ such that $Y$ consists of the $q^2(q-1)$ points of $\mathsf{Q}^-(5,q)\setminus H$ that are collinear to $p$.
\end{theorem}

\section{Preliminaries}
Let $\Gamma=(X,R)$ be a finite connected regular graph, and we will also assume throughout that $\Gamma$ is nontrivial; neither empty nor complete. A subset $Y$ of $X$ is called an {\sl intriguing set} if and only if there are integers $h_1,h_2\ge0$ such that every vertex of $Y$ is adjacent to exactly $h_1$ vertices of $Y$ and every vertex of $X\setminus Y$ is adjacent to exactly $h_2$ vertices of $Y$. We will use the symbol $j$ to denote the `all-ones' row vector. We always assume a given numbering $X=\{x_1,\dots,x_n\}$ of the vertices of $X$ where $n=|X|$. The adjacency matrix of $\Gamma$ is the real $(n\times n)$-matrix $A$ with $A_{ij}=1$ when $x_i$ and $x_j$ are adjacent and $A_{ij}=0$ otherwise. Notice that $A$ is a real symmetric matrix, so the row space $\mathbb{R}^n$ decomposes into the orthogonal sum of the eigenspaces of $A$, which we also call the eigenspaces of $\Gamma$. For every subset $Y$ of $X$ we denote by $\chi_Y$ its characteristic vector, that is the real row vector of length $n$ whose $i$-th entry is 1, if $x_i\in Y$, and $0$ otherwise.
The following theorem is from \cite{De-Bruyn:2010aa}.

\begin{result}[\cite{De-Bruyn:2010aa}]\label{bart}
Let $\Gamma=(X,R)$ be a finite connected nontrivial regular graph with eigenspaces $V_0,\dots,V_s$ where $V_0=\erz{j}$ and adjacency matrix $A$. Let $Y$ be a non-empty proper subset of $X$.
\begin{enumerate}[(i)]
\item $Y$ is an intriguing set of $\Gamma$ if and only if $\chi_Y\in V_0\oplus V_i$ for some integer $1\le i\le s$.
\item If $Y$ is an intriguing set, $i\ge 1$ is the integer with $\chi_Y\in V_0\oplus V_i$, and if $\theta$ is the eigenvalue of $A$ on $V_i$, then every vertex of $X\setminus Y$ is adjacent to exactly
    \[
    \frac{k-\theta}{|X|}\cdot |Y|
    \]
    elements of $Y$ and every vertex of $Y$ is adjacent to exactly
    \[
    \theta+\frac{k-\theta}{|X|}\cdot |Y|
    \]
    elements of $Y$.
\end{enumerate}
\end{result}

If $Y$ is an intriguing set with $Y\ne \varnothing,X$ and $i$ is the unique index such that $\chi_Y\in V_0\oplus V_i$, then we say that $Y$ is an intriguing set of \emph{type} $i$, or \emph{for the eigenspace} $V_i$. Notice that $\chi_Y=cj+v$ with $c:=|Y|/|X|$ and $v\in V_i$. Since eigenvectors for distinct eigenvalues have inner product zero, the following well-known result follows easily.

\begin{result}\label{Bart2}
Let $\Gamma=(X,R)$ be a finite connected regular graph. Let $Y_1$ and $Y_2$ be intriguing sets of $\Gamma$ for different eigenspaces. Then $|Y_1\cap Y_2|\cdot|X|=|Y_1|\cdot |Y_2|$.
\end{result}

We give an example that will be used later.

\begin{example}\label{exampleQminus5q}
The collinearity graph of $\mathsf{Q}^-(5,q)$ is a strongly regular graph with eigenvalues $(q^2+1)q$, $q-1$ and $-(q^2+1)$ where $(q^2+1)q$ has eigenspace $V_0:=\erz{j}$. Intriguing sets for the eigenvalue $-(q^2+1)$ are
hemisystems (by Segre's theorem), that is,
sets of points of size $(q^3+1)(q+1)/2$ such that every generator meets in $(q+1)/2$
elements. The intriguing sets for the eigenvalue $q-1$ are usually called `tight sets' in the literature. It follows from the above results that for any tight set $Y$ we have $|Y|=c(q+1)$, where $c$ is the number of points in $Y$ collinear to a given point not in $Y$. There exist many different examples of tight sets of $\mathsf{Q}^-(5,q)$. The $q+1$ points of a line form a tight set. The union of two perpendicular conics is a tight set. Moreover, any union of disjoint tight sets is a tight set, so a large number of examples can be obtained by taking a disjoint union of lines and pairs of perpendicular conics. If $c$ is sufficiently small, it was proven in \cite[Theorem 2.15]{Metsch:2016aa} that this construction describes all tight sets of size $c(q+1)$.
\end{example}

\section{The Penttila--Williford scheme}\label{section:PW}

We consider a finite $5$-dimensional projective space $\calP=\PG(5,q)$, $q>2$, and in there, an elliptic quadric $\calQ=\mathsf{Q}^-(5,q)$. We let $\perp$ be the associated polarity of $\calP$ whose absolute points are the points of $\calQ$. Let $H$ be a hyperplane of $\calP$ meeting $\calQ$ in a parabolic quadric $\mathsf{Q}(4,q)$.

We consider the point set $X:=\calQ\setminus (\calQ\cap H)$ obtained from $\calQ$ by removing the points of $H\cap \calQ$. Note that $|X|=q^2(q^2-1)$. We define $\sigma\colon X\to X$ as follows. Let $p$ be a point of $\calQ$
not in $H$. Then the line joining $H^\perp$ and $p$ is a hyperbolic line and so contains a unique second point $p'$ of $\calQ$.
Let $\sigma$ be the central collineation of $\PG(5,q)$ having axis $H$ and
centre $H^\perp$, mapping $p$ to $p'$. Then $\sigma$ commutes with the polarity $\perp$ and so stabilises $\calQ$.
Moreover, since $pp'$ has only two points of $\calQ$ on it, we have $\sigma(p')=p$ and hence $\sigma^2$ is the identity.
Via the Klein correspondence, we can map the points of $X$ to lines of the Hermitian surface $\mathsf{H}(3,q^2)$ that lie outside of a symplectic subgeometry $\mathsf{W}(3,q)$. This subgeometry can alternatively be defined as the fixed subspaces of a Baer involution, and this involution corresponds to our map $\sigma$.

Two points $p,p'$ of $\calQ$ are said to be \emph{collinear} when $p\ne p'$ and $p'\in p^\perp$, that is, $p$ and $p'$ are distinct points that span a line of $\PG(5,q)$ contained in $\calQ$. In this case we write $p\sim p'$. On $X$ the following symmetric relations $R_i$, $0\le i\le 4$, define a cometric association scheme on $X$
(see \cite[Theorem 1]{Penttila:2011aa} and \cite[Section 5.4]{Lee:2016aa}).
\begin{align*}
R_0&=\{(u,v)\in X\times X\mid u=v\},
\\
R_1&=\{(u,v)\in X\times X\mid v\not\sim u\sim v^\sigma\},
\\
R_2&=\{(u,v)\in X\times X\mid v\not\sim u\not\sim v^\sigma\},
\\
R_3&=\{(u,v)\in X\times X\mid v\sim u\not\sim v^\sigma\},
\\
R_4&=\{(u,v)\in X\times X\mid u=v^\sigma\}.
\end{align*}

We will frequently refer to the `association scheme $X$', which will mean the association scheme on $X$ given by the above relations. The matrix of eigenvalues $P$ and matrix of dual eigenvalues $Q$ of this association scheme are as follows:

\begin{align*}
P=&
\begin{pmatrix}1 & \left( q-1\right) \,\left( {q}^{2}+1\right)  & \left( q-2\right) \,q\,\left( {q}^{2}+1\right)  & \left( q-1\right) \,\left( {q}^{2}+1\right)  & 1\cr 1 & {q}^{2}+1 & 0 & -\left( {q}^{2}+1\right)  & -1\cr 1 & q-1 & -2\,q & q-1 & 1\cr 1 & -\left( q-1\right)  & 0 & q-1 & -1\cr 1 & -{\left( q-1\right) }^{2} & 2\,\left( q-2\right) \,q & -{\left( q-1\right) }^{2} & 1\end{pmatrix},
\\
Q=&\begin{pmatrix}1 & \frac{{\left( q-1\right) }^{2}\,q}{2} & \frac{\left( q-2\right) \,\left( q+1\right) \,\left( {q}^{2}+1\right) }{2} & \frac{\left( q-1\right) \,q\,\left( {q}^{2}+1\right) }{2} & \frac{q\,\left( {q}^{2}+1\right) }{2}\cr 1 & \frac{\left( q-1\right) \,q}{2} & \frac{\left( q-2\right) \,\left( q+1\right) }{2} & -\frac{\left( q-1\right) \,q}{2} & -\frac{\left( q-1\right) \,q}{2}\cr 1 & 0 & -q-1 & 0 & q\cr 1 & -\frac{\left( q-1\right) \,q}{2} & \frac{\left( q-2\right) \,\left( q+1\right) }{2} & \frac{\left( q-1\right) \,q}{2} & -\frac{\left( q-1\right) \,q}{2}\cr 1 & -\frac{{\left( q-1\right) }^{2}\,q}{2} & \frac{\left( q-2\right) \,\left( q+1\right) \,\left( {q}^{2}+1\right) }{2} & -\frac{\left( q-1\right) \,q\,\left( {q}^{2}+1\right) }{2} & \frac{q\,\left( {q}^{2}+1\right) }{2}\end{pmatrix}.
\end{align*}

If $A_i$ denotes the adjacency matrix of the graph $(X,R_i)$, this means that we can number the eigenspaces $V_0,\dots,V_4$ of the association scheme in such a way that $A_jv=P_{ij}v$ for $v\in V_i$, where the entries of $P$ are numbered $P_{ij}$ with $i,j=0,\dots,4$. Notice that this implies that $V_0=\erz{\chi_X}$.

The reason to exclude the case when $q=2$ is that the relation $R_2$ is empty when $q=2$ as can be seen from the middle entry in the first line of $P$.

Let $Y$ be a non-empty proper subset of $X$. We call $Y$ an \emph{intriguing set} of this association scheme, if $\chi_Y\in V_0\oplus V_i$ for some $i\in\{1,2,3,4\}$. Result \ref{bart} shows that an intriguing set $Y$ of the association scheme is an intriguing set of each graph $(V,R_i)$.

\begin{example}[Intriguing sets of type 1]
It was shown in \cite{Bamberg:2017aa} that an intriguing set of type 1, if it exists, is a relative hemisystem and $q$ is even. Moreover, Penttila and Williford \cite{Penttila:2011aa} provided an infinite family of examples that exist for each even prime power $q>2$. Further examples were constructed in
\cite{Bamberg:2016aa,Cossidente:2013aa,Cossidente:2015aa,Cossidente:2014aa}.
\end{example}

\begin{example}[Some intriguing sets of type 4]\label{examplecollineartopoint}
Let $w$ be a point of $H\cap \calQ$ and let $Y$ be the set consisting of all points of $X$ that are in $\calQ$ collinear to $w$. Then $|Y|=(q^2-q)q$, every point of $Y$ is collinear to $q-1$ points of $Y$ and every point of $X\setminus Y$ is collinear to $q^2-q$ points of $Y$. Hence $Y$ is an intriguing set of $(X,R_3)$ corresponding to the eigenvalue $q-1-(q^2-q)=-(q-1)^2$. Since $V_4$ is the eigenspace of this graph for the eigenvalue $-(q-1)^2$, it follows that $\chi_Y\in V_0\oplus V_4$, so $Y$ is also an intriguing set of the association scheme $X$.
\end{example}

Examples of intriguing sets of types 2 and 3 will appear in Section \ref{section:tight}. In the following result, we show that an intriguing set is either $\sigma$-invariant or has its image under $\sigma$ disjoint from it.

\begin{theorem}\label{hemi}
Suppose that $Y$ is a subset of $X$ such that $\chi_Y\in V_0\perp V_i$ for some $i\in\{1,2,3,4\}$.
\begin{enumerate}[(a)]
\item If $i=1$ or $i=3$, then $|Y|=\frac12|X|$ and $|\{p,p^\sigma\}\cap Y|=1$ for all $p\in X$.
\item If $i=2$ or $i=4$, then $Y^\sigma=Y$.
\end{enumerate}
\end{theorem}

\begin{proof}
Let $\theta$ be the eigenvalue of the relation $R_3$ on $V_i$. Result \ref{bart} shows that every point of $X\setminus Y$ is in relation $R_3$ to $a_3:=\frac{|Y|}{|X|}(k-\theta)$ points of $Y$ (where $k$ is the valency of $R_3$), and every point of $Y$ is in relation $R_3$ to $b_3:=a_3+\theta$ points of $Y$. Inspection of the first matrix of eigenvalues shows that $(-1)^i\theta$ is the eigenvalue of $R_1$ on $V_i$. Result \ref{bart} thus shows that every point of $X\setminus Y$ is in relation $R_1$ to $a_1:=\frac{|Y|}{|X|}(k-(-1)^i\theta)$ points of $Y$, and every point of $Y$ is in relation $R_1$ to $b_1:=a_1+(-1)^i\theta$ points of $Y$. Hence, if $p\in X\setminus Y$, then $p^\sigma$ is collinear to $a_1$ points of $Y$, and for every point $p\in Y$, the point $p^\sigma$ is collinear to exactly $b_1$ points of $Y$. Since $\theta\ne 0$, we have $a_1\ne b_1$ and $a_3\ne b_3$.

If $i$ is even, that is $i=2$ or $i=4$, then $a_1=a_3$ and $b_1=b_3$, and it follows immediately that for every point $p$, both points $p$ and $p^\sigma$ are in relation $R_3$ to the same number of points of $Y$, hence both points lie in $Y$ or both points do not. Hence $Y=Y^\sigma$ in this case.

If $i$ is odd, then $\theta=-(q^2+1)$ or $\theta=q-1$ and an easy calculation gives $a_1=b_3$ and $a_3=b_1$. This time it follows for every point $p\in X$ that $p$ and $p^\sigma$ are in relation $R_3$ to a different number of points of $Y$ and hence, exactly one of these two points lies in $Y$.
\end{proof}

As a corollary, we obtain the following result of \cite{Bamberg:2017aa}.

\begin{theorem}
If $Y$ is a non-trivial $m$-cover of $H(3,q^2)\setminus W(3,q)$, then $q$ is even, $m=\frac q2$, and from any pair of conjugate lines exactly one lies in $Y$.
\end{theorem}

\begin{proof}
We do this in the dual setting, that is we consider a set $Y$ of points of $\mathsf{Q}^-(5,q)\setminus \mathsf{Q}(4,q)$ such that every line of $\mathsf{Q}^-(5,q)$ that is not contained in $\mathsf{Q}(4,q)$ has exactly $m$ points in $Y$. Hence, in the graph $(X,R_3)$, every point of $Y$ is collinear to exactly $ (q^2+1)(m-1)$ points of $Y$, and every point of $X\setminus Y$ is collinear to exactly $(q^2+1)m$ points of $Y$. Therefore Result \ref{bart} implies that $Y$ is an intriguing set for the eigenvalue $-(q^2+1)$ of $(X,R_3)$. Since the eigenspace of the graph $(X,R_3)$ for this eigenvalue is $V_1$, which is an eigenspace of the association scheme, it follows that $Y$ is an intriguing set of the scheme. Theorem \ref{hemi} shows that $|Y|=|X|/2$ and that $Y$ contains exactly one point of every pair of conjugate points.
Since each point of $X\setminus Y$ is collinear to $m(q^2+1)$ points of $Y$, part (ii) of Result \ref{bart} gives $m=q/2$.
\end{proof}

\section{A connection to tight sets of $\mathsf{Q}^-(5,q)$}\label{section:tight}

In this section, we consider intriguing sets of the graph $(X,R_3)$ corresponding to the eigenvalue $q-1$. Notice that the eigenspace for this eigenvalue of $(X,R_3)$ is $V_2\oplus V_3$. Hence a subset $Y$ of $X$ is an intriguing set of $(X,R_3)$ if and only if $\chi_Y\in V_0\oplus V_2\oplus V_3$. We will investigate when $Y$ is also an intriguing set of the association scheme $X$, that is when $\chi_Y\in V_0\oplus V_2$ or $\chi_Y\in V_0\oplus V_3$.

By Result \ref{bart}, every point $p\in X\setminus Y$ is collinear to exactly $\alpha:=|Y|/(q+1)$ points of $X$ and every point of $Y$ is collinear to exactly $q-1+\alpha$ points of $Y$. Recall that $X$ is the set of points of the elliptic quadric $\calQ$ that lie outside the non-tangent hyperplane $H$ of the ambient projective space.

\begin{lemma}\label{tightset}
If $Y$ is an intriguing set of the graph $(X,R_3)$ corresponding to the eigenvalue $q-1$, then $Y$ is a tight set of $\mathsf{Q}^-(5,q)$.
\end{lemma}
\begin{proof}
Let $p$ be a point of $H\cap\calQ$. We have seen in Example \ref{examplecollineartopoint} that the set $\F$ consisting of the $q^2(q-1)$ points of $X$ that are in $\calQ$ collinear to $p$ is an intriguing set of the association scheme with $\chi_\F\in V_0\oplus V_4$. Applying Result \ref{Bart2} to the intriguing sets $Y$ and $\F$ of $(X,R_3)$, it follows that
\[
|Y\cap\F|=\frac{|Y|\cdot|\F|}{|X|}=\frac{|Y|}{q+1}.
\]
Hence $p$ is collinear to exactly $|Y|/(q+1)$ points of $Y$. Since this is true for all points of $H\cap\calQ$ and since the points of $X\setminus Y$ are collinear to the same number $|Y|/(q+1)$ points of $Y$, it follows that $Y$ is an intriguing set of the collinearity graph of $\mathsf{Q}^-(5,q)$ for the eigenvalue $q-1$. Hence $Y$ is a tight set of $\mathsf{Q}^-(5,q)$; see Example \ref{exampleQminus5q}.
\end{proof}

\begin{remark} The algebraic reason behind this phenomenon is the following.
If $\calP$ is the set of points of $\mathsf{Q}^-(5,q)$, then we have the following decomposition into eigenspaces of the collinearity relation:
\begin{center}
\begin{tabular}{c|c}
\toprule
$\mathbb{C}\calP=\langle \chi_{\calP}\rangle\oplus V^+\oplus V^-$&
\begin{tabular}{cccc}
Eigenspace & $\langle \chi_{\calP}\rangle$ & $V^+$ & $V^-$\\
Dimension & 1& $q^2(q^2+1)$ &  $q(q^2-q+1)$\\
\end{tabular}\\
\bottomrule
\end{tabular}
\end{center}
Whereas for the association scheme on $X$, we have
\begin{center}
\setlength{\tabcolsep}{4pt}
\begin{tabular}{c|c}
\toprule
$\mathbb{C}X=\langle \chi_{X}\rangle\oplus V_1\oplus V_2\oplus V_3 \oplus V_4$&
\begin{tabular}{cccccccc}
Eigenspace & $\langle \chi_{X}\rangle$ & $V_1$ & $V_2$ & $V_3$& $V_4$\\
Dimension &1 & $\frac{{q\left( q-1\right) }^{2}}{2}$ & $\frac{\left( q-2\right) \left( q+1\right) \left( {q}^{2}+1\right) }{2}$ & $\frac{\left( q-1\right) q\left( {q}^{2}+1\right) }{2}$ & $\frac{q\left( {q}^{2}+1\right) }{2}$
\end{tabular}\\
\bottomrule
\end{tabular}
\end{center}
We have two linear maps: (1) the natural inclusion map $\iota\colon \mathbb{C}X\to \mathbb{C}\calP$; (2) the projection map
$\rho\colon\mathbb{C}\calP\to \mathbb{C}X$ which maps $\sum_{x\in \calP} \alpha_x \chi_x$ to  $\sum_{x\in X} \alpha_x \chi_x$.
Note that $\iota$ is injective, $\rho$ is surjective, and $\rho\circ\iota$ is the identity on $\mathbb{C}X$.
The kernel of $\rho$ is clearly $K:=\langle \chi_h\colon h\in H\cap\calP\rangle$. Let $S$ be the $|X|\times |\calP|$ inclusion matrix
for $X$ in $\mathcal{P}$.
Let $(A_0,A_1,A_2,A_3,A_4)$ be the adjacency matrices for the association scheme on $X$, and
let $(B_0,B_1,B_2)$ be the adjacency matrices for the association scheme on $\mathcal{P}$, where
$B_1$ represents the collinearity relation. So $SB_0S^\top =A_0$, $SB_1S^\top =A_3$, and $SB_2S^\top=A_1+A_2+A_4$.

From the matrix of dual eigenvalues\footnote{which is
$\begin{bmatrix}
 1 & q^2(q^2+1) & q (q^2-q+1) \\
 1 & q(q-1) & -(q^2-q+1) \\
 1 & -\tfrac{1}{q}(q^2+1) & \tfrac{1}{q}(q^2-q+1) \\
\end{bmatrix}$
} for the association scheme on $\mathbb{C}\mathcal{P}$, we have
\[
(q+1)^2 E^-=  q B_0 -B_1+\tfrac{1}{q} B_2
\]
and hence
\begin{align*}
SE^-S^\top=&\frac{1}{(q+1)^2}\left(q A_0-A_3+\tfrac{1}{q}(A_1+A_2+A_4)\right)\\
=&\frac{1}{q^2(q^2-1)}\left(\frac{q(q-1)^2}{2}A_0+\frac{q(q-1)}{2}A_1-\frac{q(q-1)}{2}A_3 -\frac{q(q-1)^2}{2}A_4\right)\\
&+\frac{1}{ q^2 (q+1)^2}\left(\frac{q(q^2+1)}{2}A_0-\frac{q(q-1)}{2}A_1+qA_2-\frac{q(q-1)}{2}A_3+\frac{q(q^2+1)}{2}A_4\right)\\
=&E_1+\tfrac{q-1}{ q+1}E_4.
\end{align*}

Therefore, $\rho(V^-)$ contains the rowspace of $E_1+\tfrac{q-1}{ q+1}E_4$. Since $E_1E_4=0$, the rowspace
of $E_1+\tfrac{q-1}{ q+1}E_4$ is the sum of the rowspaces of $E_1$ and $E_4$; that is, $V_1\perp V_4$. Hence $\rho(V^-)$ contains
$V_1\perp V_4$. As $\dim(V^-)=\dim(V_1)+\dim(V_4)$, it follows that $\rho$ restricted to $V^-$ is a bijective map from $V^-$
onto $V_1\perp V_4$.

For each $v\in \langle \chi_{X}\rangle\perp V_2\perp V_3$, we have
\begin{align*}
(\iota(v) E^-)(\iota(v) E^-)^\top&= \iota(v) (E^-)^2 (\iota(v))^\top= \iota(v) E^- (\iota(v))^\top\\
&=v S E^-S^\top v^\top=v (E_1+\tfrac{q-1}{ q+1} E_4)v^\top\\
&=0.
\end{align*}
Therefore, $\iota(v) E^-=0$ and consequently
$\iota(v)\in \langle \chi_\calP\rangle \perp V^+$. Hence $\iota$ maps $\langle \chi_{X}\rangle\perp V_2\perp V_3$ into
$\langle \chi_\calP\rangle \perp V^+$, which affirms Lemma \ref{tightset}.
So we have the following diagram:
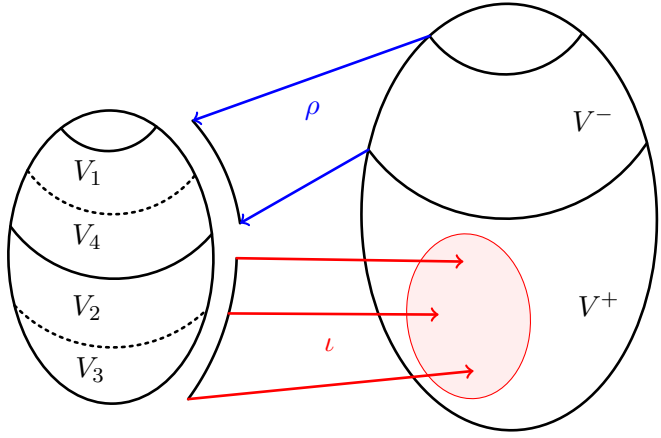
\begin{center}
\definecolor{qqqqff}{rgb}{0.,0.,1.}
\definecolor{ffqqqq}{rgb}{1.,0.,0.}
\begin{tikzpicture}[line cap=round,line join=round,x=1.0cm,y=1.0cm,scale=0.75]
\draw [rotate around={-89.07595464722672:(-0.37,1.68)},line width=1.pt] (-0.37,1.68) ellipse (2.600270691056427cm and 1.8168400223374284cm);
\draw [rotate around={-88.95074405911865:(6.69,2.39)},line width=1.pt] (6.69,2.39) ellipse (3.782757374478706cm and 2.6179865076377014cm);
\draw [shift={(-0.4195190875247509,4.619064395562453)},line width=1.pt]  plot[domain=3.8053015626212314:5.65413241769698,variable=\t]({1.*1.059244253572264*cos(\t r)+0.*1.059244253572264*sin(\t r)},{0.*1.059244253572264*cos(\t r)+1.*1.059244253572264*sin(\t r)});
\draw [shift={(-0.3307875948613303,4.314028114376794)},line width=1.pt,dotted]  plot[domain=3.768181030912218:5.601482666649646,variable=\t]({1.*1.874674167500628*cos(\t r)+0.*1.874674167500628*sin(\t r)},{0.*1.874674167500628*cos(\t r)+1.*1.874674167500628*sin(\t r)});
\draw [shift={(-0.3101343264651509,3.5830193063298212)},line width=1.pt]  plot[domain=3.777278968516448:5.570656425145259,variable=\t]({1.*2.287932882140286*cos(\t r)+0.*2.287932882140286*sin(\t r)},{0.*2.287932882140286*cos(\t r)+1.*2.287932882140286*sin(\t r)});
\draw [shift={(-0.33665675464156547,2.4870421155668465)},line width=1.pt,dotted]  plot[domain=3.9051834010555897:5.478550662404903,variable=\t]({1.*2.410200261863746*cos(\t r)+0.*2.410200261863746*sin(\t r)},{0.*2.410200261863746*cos(\t r)+1.*2.410200261863746*sin(\t r)});
\draw [shift={(6.615840236808241,6.572643646811485)},line width=1.pt]  plot[domain=3.7687178350780366:5.690727620233133,variable=\t]({1.*1.652820231454058*cos(\t r)+0.*1.652820231454058*sin(\t r)},{0.*1.652820231454058*cos(\t r)+1.*1.652820231454058*sin(\t r)});
\draw [shift={(6.621813513362459,5.379109542870281)},line width=1.pt]  plot[domain=3.778466454948137:5.694438151857333,variable=\t]({1.*3.0196701970708033*cos(\t r)+0.*3.0196701970708033*sin(\t r)},{0.*3.0196701970708033*cos(\t r)+1.*3.0196701970708033*sin(\t r)});
\draw (-1.22,3.58) node[anchor=north west] {$V_1$};
\draw (-1.24,1.18) node[anchor=north west] {$V_2$};
\draw (-1.18,0.08) node[anchor=north west] {$V_3$};
\draw (-1.24,2.44) node[anchor=north west] {$V_4$};
\draw (7.62,4.52) node[anchor=north west] {$V^-$};
\draw (7.72,1.24) node[anchor=north west] {$V^+$};
\draw [shift={(-1.555,1.78)},line width=1.pt]  plot[domain=0.14290410360914918:0.7219114475319262,variable=\t]({1.*3.5107869488193066*cos(\t r)+0.*3.5107869488193066*sin(\t r)},{0.*3.5107869488193066*cos(\t r)+1.*3.5107869488193066*sin(\t r)});
\draw [shift={(-2.604090909090911,1.7977272727272737)},line width=1.pt]  plot[domain=5.651388498502753:6.252342835247914,variable=\t]({1.*4.466215002245303*cos(\t r)+0.*4.466215002245303*sin(\t r)},{0.*4.466215002245303*cos(\t r)+1.*4.466215002245303*sin(\t r)});
\draw [rotate around={-84.69907348347446:(5.97,0.63)},line width=0.4pt,color=ffqqqq,fill=ffqqqq,fill opacity=0.07000000029802322] (5.97,0.63) ellipse (1.4646315666641523cm and 1.0936844270945356cm);
\draw [->,line width=1.pt,color=ffqqqq] (1.72,0.68) -- (5.44,0.66);
\draw [->,line width=1.pt,color=qqqqff] (5.2775221905205525,5.602736679889115) -- (1.08,4.1);
\draw [->,line width=1.pt,color=qqqqff] (4.194123124779036,3.5833569217508536) -- (1.92,2.28);
\draw [->,line width=1.pt,color=ffqqqq] (1.86,1.66) -- (5.88,1.6);
\draw [->,line width=1.pt,color=ffqqqq] (1.,-0.84) -- (6.06,-0.34);
\draw[color=ffqqqq] (3.5,0.13) node {$\iota$};
\draw[color=qqqqff] (3.2,4.27) node {$\rho$};
\end{tikzpicture}
\end{center}
\end{remark}

Now we investigate when $Y$ is an intriguing set of the association scheme $X$, that is when $\chi_Y\in V_0\perp V_2$ or $\chi_Y\in V_0\perp V_3$. We already know that $\chi_Y\in V_0\oplus V_2$ implies $Y=Y^\sigma$ and that $\chi_Y\in V_0\oplus V_3$ implies $Y\cap Y^\sigma=\varnothing$. The following observation is crucial in proving that the converse is also true.

\begin{lemma}\label{sigmaclosed}
We have $\erz{\chi_p - \chi_{p^\sigma}\mid p\in X}=V_1\oplus V_3$ and
$\erz{\chi_p + \chi_{p^\sigma}\mid p\in X}=V_0\oplus V_2\oplus V_4$.
\end{lemma}
\begin{proof}
We number the rows and columns of $Q$ from $0$ to $4$. Inspection of $Q$ shows that $Q_{ij}=(-1)^jQ_{4-i,j}$ for all $i,j=0,\dots,4$. The definition of the relations $R_i$ shows that $A_i\chi_{p^\sigma}=A_{4-i}\chi_p$. It follows for each $j\in\{0,1,2,3,4\}$ that
\begin{align*}
E_j\chi_{p^\sigma}=\frac{1}{|X|}\sum_iQ_{ij}A_i\chi_{p^\sigma}
=\frac{1}{|X|}\sum_i(-1)^jQ_{ij}A_{4-i}\chi_p=(-1)^jE_j\chi_p.
\end{align*}
Hence, for each $p\in X$ we have $E_j(\chi_p-\chi_{p^\sigma})=0$ for $j=0,2,4$ and $E_j(\chi_p+\chi_{p^\sigma})=0$ for $j=1,3$. Hence
$\erz{\chi_p - \chi_{p^\sigma}\mid p\in X}\subseteq V_1\oplus V_3$ and
$\erz{\chi_p + \chi_{p^\sigma}\mid p\in X}\subseteq V_0\oplus V_2\oplus V_4$. Since $\{\chi_p\mid p\in X\}$ forms a basis of $V_0\oplus V_1\oplus V_2\oplus V_3\oplus V_4$, we have equality.
\end{proof}

\begin{lemma}\label{sigmaclosed2}
If $S\subseteq X$, then $S=S^\sigma$ if and only if $\chi_S\in V_0\oplus V_2\oplus V_4$.
\end{lemma}
\begin{proof}
We have $S=S^\sigma$ if and only if
\[
(\chi_p-\chi_{p^\sigma})\cdot \chi_S=0
\]
for all $p\in\calP$. Therefore the statement follows from Lemma \ref{sigmaclosed}.
\end{proof}

\begin{theorem}\label{V2V3characterization}
Let $Y$ be a proper nonempty subset of $X$ with $\chi_Y\in V_0\oplus V_2\oplus V_3$, that is, $Y$ is an intriguing set of $(X,R_3)$. Then $\chi_Y\in V_0\oplus V_2$ if and only if $Y^\sigma=Y$, and $\chi_Y\in V_0\oplus V_3$ if and only if $Y$ contains exactly one of every two conjugate points of $X$.
\end{theorem}
\begin{proof}
If we assume that $\chi_Y\in V_0\oplus V_2$ or $\chi_Y\in V_0\oplus V_3$, then Theorem \ref{hemi} proves the statement. If we assume that $Y^\sigma=Y$, then $\chi_Y\in V_0\oplus V_2\oplus V_4$ by Lemma \ref{sigmaclosed2} and hence $\chi_Y\in V_0\oplus V_2$. Finally assume that $Y$ contains one point of any pair of conjugate points. Then $(\chi_p+\chi_{p^\sigma})\cdot(2\chi_Y-j)=0$ for every point $p\in X$. Therefore Lemma \ref{sigmaclosed} shows that $2\chi_Y-j\in V_1\oplus V_3$. Hence $2\chi_Y-j\in V_3$, which shows that $\chi_Y\in V_0\oplus V_3$.
\end{proof}

\subsection*{Intriguing sets of type $3$} %\label{section:type3}

We construct an intriguing set $Y$ of the association scheme $X$ with $\chi_Y\in V_0\oplus V_3$. Recall that $X=\calQ\setminus (H\cap\calQ)$ where $\calQ$ is an elliptic quadric $\calQ$ in $\PG(5,q)$ and $H$ is a non-tangent hyperplane of $\PG(5,q)$, that is a hyperplane that meets $\calQ$ in a parabolic quadric $\mathsf{Q}(4,q)$.

\begin{example}[Intriguing sets of type $3$, $q$ even]\label{example:type3_qeven}\ \\
Assume $q$ is even and $q>2$. Consider a solid $S$ of $H$ such that $S\cap \calQ$ is a hyperbolic quadric. Then $S^\perp$ is an elliptic line of $\PG(5,q)$ and $H^\perp$ is a point on this line.
Since $q$ is even, $H^\perp\in H$ and $\sigma$ has $q/2$ orbits of length two in its action on the remaining $q$ points of $S^\perp$.

Consider $p\in S^\perp$ with $p\ne H^\perp$. Then $p^\perp$ is a non-tangent hyperplane to $\calQ$ and hence $p^\perp\cap \calQ$ is a parabolic quadric $Q(4,q)$ of $\calQ$. This implies that $p^\perp\cap \calQ$ is a tight set of $\calQ$. As $p\in S^\perp$ we have $p^\perp\cap H=S$. As $S\cap \calQ$ is a hyperbolic quadric $\mathsf{Q}^+(3,q)$, which is the union of its lines of each of its reguli, the set $S\cap \calQ$ is also a tight set of $\calQ$ and it is contained in $p^\perp\cap \calQ$. Hence $T_p:=(p^\perp\setminus S)\cap \calQ$ is a tight set of $\calQ$. For different points $p$ the different hyperplanes $p^\perp$ meet within $S$ and hence the corresponding tight sets $T_p$ are disjoint. As $\sigma$ has order two and $H^\perp$ is the only point of $S^\perp$ fixed by $\sigma$, we can partition the set of $q$ points of $S^\perp$ into two sets $M$ and $M'$ of cardinality $q/2$ such that $M^\sigma=M'$. Then the union $T_M$ of the tight sets $T_p$ with $p\in M$ is a tight set of $\calQ$ and the union $T_{M'}$ of the tight sets $T_p$ with $p\in M'$ is a tight set of $\calQ$. Moreover every point of $X$ is contained in exactly one of the two sets $T_M$ and $T_{M'}$. Since $M^\sigma=M'$, we have $(T_M)^\sigma=T_{M'}$. Theorem \ref{V2V3characterization} now shows that $T_M$ is an intriguing set of $X$ with $\chi_{T_M}\in V_0\oplus V_3$.
\end{example}

\begin{example}[Intriguing sets of type $3$, $q$ odd]\label{example:type3_qodd}
Here we assume that $q$ is odd, and hence $H^\perp\notin H$. Let $m$ be a line of $H\cap \calQ$. Then $H^\perp$ and $m$ generate a plane $\pi$ that meets the quadric $\calQ:=\calQ$ only in the line $m$. Set $z=H^\perp$ and let $u$ be a point of $m$. Then $(zu)^\perp$ is a 3-space that is contained in $H$ and meets $\calQ$ in a cone $C$ with vertex $u$ over a conic. In particular $|C|=q^2+q+1$.

Let $S_1$ be any set consisting of $\frac12(q-1)$ points of the line $uz$ such that $u,z\notin S_1$, let $S_2$ be any set consisting of $\frac12(q-1)$ lines of $\pi$ on $u$ such that $uz,m\notin S_2$, and let $T$ be the set consisting of all points of $\calQ\setminus H$ that are perpendicular to a point in $S_1$ or to a line in $S_2$. For $y\in S_1$, the hyperplane $y^\perp$ meets $\calQ$ in a quadric $\mathsf{Q}(4,q)$, and $y^\perp\cap H$ meets $\calQ$ in the cone $C$. For $\ell\in S_2$, the solid $\ell^\perp$ meets $\calQ$ in a cone with vertex $u$ over a conic, and $\ell^\perp\cap H$ meets $\calQ$ in the line $m$. It follows that
\[
|T|=\frac{q-1}{2}(q^3+q^2+q+1-|C|)+\frac{q-1}{2}(q^2+q+1-|m|)=\frac12q^2(q^2-1).
\]
Let $p$ be a point of $X$. We want to determine $|p^\perp\cap T|$. First notice that $p^\perp$ meets the line $uz$ in exactly one point and this point is different from $z$. Also, if $p^\perp\cap uz$ is the point $u$, then $p^\perp\cap \pi$ is one of the lines of $\pi$ on $u$ that is different from $m$ and $uz$. For $y\in S_1$ we have $y^\perp\cap H\cap \calQ=C$ and thus
\[
|p^\perp\cap y^\perp\cap T|=|p^\perp\cap y^\perp\cap \calQ|-|p^\perp\cap C|.
\]
For $\ell\in S_2$ we have $\ell^\perp\cap H\cap \calQ=m$ and thus
\[
|p^\perp\cap \ell^\perp\cap T|=|p^\perp\cap \ell^\perp\cap \calQ|-|p^\perp\cap m|.
\]
As $|S_1|=|S_2|=\frac12(q-1)$, it follows that
\begin{align*}
|p^\perp\cap T|
&=\sum_{y\in S_1}|p^\perp\cap y^\perp\cap T|+\sum_{\ell\in S_1}|p^\perp\cap \ell^\perp\cap T|.
\\
&=\sum_{y\in S_1}|p^\perp\cap y^\perp\cap \calQ|+\sum_{\ell\in S_1}|p^\perp\cap \ell^\perp\cap \calQ|-\frac12(q-1)(|p^\perp\cap C|+|p^\perp\cap m|).
\end{align*}

\begin{description}
\item[Case 1. $p^\perp\cap uz\notin S_1\cup\{u\}$]\ \\
Then $p^\perp\cap C$ is a conic and $p^\perp\cap m$ is a single point. If $y\in S_1$, then $y^\perp\cap \calQ$ is a $\mathsf{Q}(4,q)$ not containing $p$, so $p^\perp$ meets this $\mathsf{Q}(4,q)$ in a $\mathsf{Q}^-(3,q)$. If $\ell\in S_2$, then $\ell^\perp\cap \calQ$ is a cone with vertex $u$ over a conic and since $p\notin u^\perp$, the hyperplane $p^\perp$ meets this cone in a conic. Hence
\[
|p^\perp\cap T|=|S_1|(q^2+1)+|S_2|(q+1)-\frac12(q-1)((q+1)+1)=\frac12(q-1)q^2.
\]

\item[Case 2. $p^\perp\cap uz$ is a point $y_0$ of $S_1$]\
\\
The only difference to Case 1 occurs for the point $y_0\in S_1$, namely $p^\perp\cap y_0^\perp$ is a solid that meets $\calQ$ in a cone with vertex $p$ over a conic, so $|p^\perp\cap y_0^\perp\cap \calQ|=q^2+q+1$. It follows that $|p^\perp \cap T|=\frac12(q-1)q^2+q$.

\item[Case 3. $p^\perp\cap uz=\{u\}$ and the line $p^\perp\cap \pi$ is not in $S_2$]\
\\
Then $p^\perp\cap C=p^\perp\cap m=\{u\}$. If $y\in S_1$, then $p^\perp\cap y^\perp\cap \calQ$ is a $\mathsf{Q}^-(3,q)$ as before. As $pu$ is a line of $\calQ$ we have $(pu)^\perp\cap \calQ=pu$. If $\ell\in S_2$, then $\ell$ and $p$ are not perpendicular and therefore $\ell^\perp\cap p^\perp\cap \calQ=\ell^\perp\cap (pu)^\perp\cap\calQ=\ell^\perp\cap pu=\{u\}$. Hence $|p^\perp\cap T|=\frac12(q-1)q^2$.

\item[Case 4. $p^\perp\cap uz=\{u\}$ and the line $\ell_0:=p^\perp\cap \pi$ is a line of $S_2$]\ \\
The only difference to Case 3 is that now $p^\perp\cap \ell_0^\perp$ meets $\calQ$ in the line $pu$, so that $|p^\perp\cap \ell_0^\perp\cap\calQ|=q+1$. Hence $|p^\perp\cap T|=\frac12(q-1)q^2+q$.
\end{description}

Notice that Case 2 or 4 occurs if and only if $p\in T$. Hence
\[
|p^\perp\cap T|=\begin{cases}
\frac12(q-1)q^2&\text{if }p\notin T,\\
\frac12(q-1)q^2+q&\text{if }p\in T.
\end{cases}
\]
This shows that $T$ is a tight set of the graph $(X,R_3)$. Since $\sigma$ has orbits of length two on the points not equal to $u,z$ of $uz$, and orbits of length two on the lines not equal to $m,uz$ of $\pi$ on $u$, we can choose $S_1$ and $S_2$ in such a way that $S_1^\sigma\cap S_1=\varnothing$ and $S_2^\sigma\cap S_2=\varnothing$. As the polarity of $\calQ$ commutes with $\sigma$, this implies that $T^\sigma\cap T=\varnothing$. Since $|T|=\frac12|X|$, it follows that $|T\cap\{p,p^\sigma\}|=1$ for every point $p\in X$. Therefore Theorem \ref{V2V3characterization} shows that $\chi_T\in V_0\oplus V_3$.

\end{example}

\subsection*{Intriguing sets of type $2$}

Now we focus on the intriguing sets $Y$ with $\chi_Y\in V_0\oplus V_2$. We will prove a lower bound on the size of $Y$ and provide an example that shows that the bound is sharp.

\begin{lemma}\label{perpendicularconics}
If $S$ is the union of two perpendicular conics of $\mathsf{Q}^-(5,q)$ with $S\cap H=\varnothing$, then $S^\sigma\ne S$.
\end{lemma}
\begin{proof}
Assume $S$ is the union of conics $C_1=\pi_1\cap \mathsf{Q}^-(5,q)$ and $C_2=\pi_2\cap \mathsf{Q}^-(5,q)$ where $\pi_1$ and $\pi_2$ are perpendicular planes. Then $\pi_1\cap\pi_2=\varnothing$. Also $\pi_1\cap H$ and $\pi_2\cap H$ are elliptic lines. Suppose that $S^\sigma=S$. As $\pi_1^\sigma$ is a plane on $\pi_1\cap H$, then either $\pi_1^\sigma=\pi_1$ or $\pi_1^\sigma\cap S\subseteq \pi_1^\sigma\cap \pi_2$. As $\pi_1^\sigma\cap S=\pi_1^\sigma\cap S^\sigma=(\pi_1\cap S)^\sigma$ is a conic, the second case is impossible. Hence $\pi_1^\sigma=\pi_1$. As $\sigma$ is a non-trivial central collineation with centre $H^\perp$, it follows that $H^\perp\in\pi_1$. The same argument shows $H^\perp\in\pi_2$. But $\pi_1\cap\pi_2=\varnothing$, a contradiction.
\end{proof}

\begin{lemma}\label{alpha4}
If $Y$ is a non-empty subset of $X$ with $\chi_Y\in V_0\oplus V_2$, then $|Y|=\alpha(q+1)$ where $\alpha$ is an even integer such that $\alpha\ge 4$.
\end{lemma}
\begin{proof}
The set $Y$ is a tight set of $\mathsf{Q}^-(5,q)$ (by Lemma \ref{tightset}), and so
$|Y|=\alpha(q+1)$ for some positive integer $\alpha$.
Let $\F$ be an intriguing set of type 4 of size $q^2(q-1)$ as in Example \ref{examplecollineartopoint}. Then by Result \ref{bart},
\[
|Y\cap\F|=\frac{\alpha(q+1)\cdot q^2(q-1)}{q^2(q^2-1)}=\alpha.
\]
Now $Y$ and $\F$ are both $\sigma$-invariant (by Theorem \ref{hemi}) and so $Y\cap \F$ is $\sigma$-invariant. So $Y\cap \F$ is a union of $\sigma$-orbits, each of size 2. Therefore, $|Y\cap\F|=\alpha$ is even.

A $2$-tight set is either a union of two lines or two perpendicular conics (c.f., \cite[Theorem 1.1]{Metsch:2016aa}). As $Y$ is disjoint from the hyperplane $H$, we have that $Y$ contains no line. As $Y^\sigma=Y$, Lemma \ref{perpendicularconics} shows that $Y$ is not a union of two perpendicular conics. Hence $\alpha\ne 2$ and thus $\alpha\ge 4$.
\end{proof}

\begin{example}[Intriguing sets of type 2]\label{example:type2}
Recall that $X=\calQ\setminus H$. Consider a solid $S$ of $H$ such that $S\cap \calQ$ is a hyperbolic quadric $\mathsf{Q}^+(3,q)$. Then $S^\perp$ is an elliptic line of $\PG(5,q)$, that is a line of $\PG(5,q)$ with no point in $\calQ$, and $H^\perp$ is a point on this line. Consider an elliptic line $\ell_1$ of $S$. Then $\ell_2:=\ell_1^\perp\cap S$ is also an elliptic line. Let $p_1$ and $p_2$ be two perpendicular points of the line $S^\perp$ that are different from $H^\perp$ (and hence of $S^\perp\cap H$). Then $\erz{p_1,\ell_1}$ and $\erz{p_2,\ell_2}$ are perpendicular conic planes that have all their conic points outside $H$. Let $Y$ be the set consisting of the $2(q+1)$ conic points in these two planes. As the conics are perpendicular, $Y$ is a $2$-tight set of $\calQ$ (cf. Example \ref{exampleQminus5q}). The points $p_1^\sigma$ and $p_2^\sigma$ are perpendicular points of $S^\perp$ and as before, $\erz{p_1^\sigma,\ell_1}$ and $\erz{p_2^\sigma,\ell_2}$ are conic planes and the $2(q+1)$ points of their conics form a $2$-tight set $Y'$ of $\calQ$. We have $Y'=Y^\sigma$ and $Y=(Y')^\sigma$, so $Y\cup Y'$ is $\sigma$-invariant.  Notice that either $p_1^\sigma=p_2$ and $p_2^\sigma=p_1$ or otherwise the points $p_1$, $p_2$, $p_1^\sigma$ and $p_2^\sigma$ are four distinct points. In both cases, it is easy to see that no four of the used conic planes share a point of $\calQ$. Hence $Y\cap Y'=\varnothing$, which implies that $Y\cup Y'$ is a $4$-tight set of $\calQ$. As it is invariant under $\sigma$ and since $Y,Y'\subseteq X$, it follows from Theorem \ref{V2V3characterization} that $\chi_{Y\cup Y'}\in V_0+V_2$. Hence $Y\cup Y'$ is an intriguing set of $X$ size $4(q+1)$.
\end{example}

\begin{lemma}\label{meetsinhyperbolic}
Let $\calQ$ be a five-dimensional elliptic quadric of $\PG(5,q)$, and suppose $\ell_1$ and $\ell_2$ are two elliptic lines
(with respect to $\calQ$) that are disjoint and perpendicular. Then the solid $\langle \ell_1,\ell_2\rangle$ meets
$\calQ$ in a hyperbolic quadric.
\end{lemma}

\begin{proof}
Let $S$ be the solid spanned by $\ell_1$ and $\ell_2$. Now $\ell_2^\perp$ is equal to $\langle \ell_1, S^\perp\rangle$ because $\ell_1$ and $S^\perp$ are both contained in $\ell_2^\perp$, and $\ell_1\cap S^\perp = \ell_1\cap (\ell_1^\perp\cap\ell_2^\perp)=\varnothing$. 
Note that $\ell_2^\perp$ is a solid of hyperbolic type. 
Within the space $\ell_2^\perp$, if we consider it just as the hyperbolic quadric $\mathsf{Q}^+(3,q)$,
the line $S^\perp$ is simply the polar image of $\ell_1$, and therefore, $m$ is elliptic. Hence it follows that $S$ is hyperbolic.
\end{proof}

\begin{theorem}\label{characterise_type2}
For $q\ge 59$, every intriguing set of $Z$ with $\chi_Z\in V_0+V_2$ and $|Z|=4(q+1)$ is of the form described in Example \ref{example:type2}.
\end{theorem}
\begin{proof}
Let $Z$ be such an intriguing set (of type 2). Now $Z$ does not contain a line because $Z\cap H=\varnothing$. Since $q\ge 59$, it follows therefore from \cite[Theorem 2.15]{Metsch:2016aa} that $Z$ is the union of four conics $C_1$, $C_2$, $C_3$ and $C_4$ where $C_1$ and $C_2$, as well as $C_3$ and $C_4$, are perpendicular. Let $\pi_i$ be the plane spanned by $C_i$. Since $Z\subseteq X$ and hence $Z\cap H=\varnothing$, the lines $\ell_i:=\pi_i\cap H$, $i=1,2,3,4$, are elliptic lines. Since $\pi_1$ and $\pi_2$ are perpendicular and hence skew, the lines $\ell_1$ and $\ell_2$ are skew and perpendicular. Hence these two lines span a solid $S$ that meets $\calQ$ in a hyperbolic quadric $\mathsf{Q}^+(3,q)$ (by Lemma \ref{meetsinhyperbolic}). We have $S\subseteq H$ and $S^\perp$ is an elliptic line. Since $S^\perp$ and $\pi_1$ are contained in $\ell_2^\perp$, it follows that $\pi_1$ and $S^\perp$ meet in a point $p_1$, and similarly $\pi_2$ and $S^\perp$ meet in a point $p_2$. Lemma \ref{perpendicularconics} shows that $C_1\cup C_2$ is not invariant under $\sigma$, and Theorem \ref{V2V3characterization} shows that $Z^\sigma=Z$. Since $C_1\cup C_2\subset Z$, it follows that $\pi_1^\sigma\ne \pi_1$ or $\pi_2^\sigma\ne \pi_2$. We may assume without loss of generality that $\pi_1^\sigma\ne \pi_1$. Then $\pi_1^\sigma$ is a conic plane on $\ell_1$. Since the conic $\pi_1^\sigma\cap\calQ$ is contained in $Z^\sigma=Z$ and hence in the union of the conics $C_2$, $C_3$ and $C_4$, it follows from $q\ge 59$ that $\pi_1^\sigma$ must be one of the planes $\pi_2$, $\pi_3$ or $\pi_4$. Now $\ell_1\subseteq \pi_1^\sigma$, and so $\pi_1^\sigma\ne \pi_2$ and hence we may assume that $\pi_1^\sigma=\pi_3$. Therefore, $\pi_2^\sigma=(\pi_1^\perp)^\sigma=(\pi_1^\sigma)^\perp=\pi_3^\perp=\pi_4$. Thus with $Y:=C_1\cup C_2$ and $Y'=C_3\cup C_4$ we have $Y^\sigma=Y'$. So, $Z=Y\cup Y'$ has the structure described in Example \ref{example:type2}.
\end{proof}

\section{Intriguing sets that are not tight; type 4}\label{type4}

Let $Y$ be an intriguing set of the association scheme $X$ corresponding to $V_4$. Since the relation $R_3$ has the eigenvalue $\theta:=-(q-1)^2$ on $V_4$, Result \ref{bart} shows that every point of $X\setminus Y$ is collinear to
\[
\alpha:=\frac{k-\theta}{|L|}\cdot |Y|=\frac{1}{q}|Y|
\]
points of $Y$, and every point of $Y$ is collinear to exactly $\theta+\alpha$ points of $Y$. Since $\alpha+\theta\ge 0$, we have $\alpha\ge -\theta=(q-1)^2$ and hence $|Y|\ge q(q-1)^2$.

\begin{lemma}
Let $Y$ be an intriguing set of type $4$. Then $|Y|\ge q^2(q-1)$.
\end{lemma}
\begin{proof}
We have seen in Example \ref{example:type2} that there exists an intriguing set $S$ of size $4(q+1)$ with $\chi_S\in V_0\oplus V_2$. Hence
\begin{align*}
|S\cap Y|=\frac{|S|\cdot|Y|}{|X|}=\frac{4(q+1)|Y|}{q^2(q^2-1)}=\frac{4|Y|}{q^2(q-1)}.
\end{align*}
that is $q^2(q-1)$ divides $4|Y|$. We also have $|Y|\ge q(q-1)^2$. As $q\ge 3$, it follows that either $|Y|\ge q^2(q-1)$ or otherwise that $q=4$ and $|Y|=q(q-1)^2=36$. However, $|Y|=q^2(q-1)$ and $q=4$ implies that $\alpha+\theta=0$, that is $Y$ is a partial ovoid. It is known that $\mathsf{Q}^-(5,4)$ has no partial ovoids of size larger than $35$ \cite[Theorem 4.2]{De-Beule:2008aa}.
\end{proof}

We have seen in Example \ref{examplecollineartopoint} that this bound is tight. In the rest of this section we show that every intriguing set $Y$ of type 4 with $|Y|=q^2(q-1)$ has the form described in Example \ref{examplecollineartopoint}. Recall that $H$ is the hyperplane of $\PG(5,q)$ such that $X=\calQ\setminus H$, and that $H\cap \calQ$ is a parabolic quadric $\mathsf{Q}(4,q)$.

\begin{lemma}\label{properties}
Suppose that $Y$ is an intriguing set of $X$ with $q^2(q-1)$ points such that every point in $Y$ is collinear to $q-1$ points of $Y$, and such that every point of $X\setminus Y$ is collinear to $q^2-q$ points of $Y$. Then we have.
\begin{enumerate}
\renewcommand{\labelenumi}{\rm(\alph{enumi})}
\item For each $p\in H\cap \calQ$ there exists an integer $s_p$ such that $|\ell\cap Y|=s_p$ for every line of $\calQ$ that contains $p$ and is not contained in $H$. For every subset $M$ of $H\cap \calQ$ we set $s_M:=\sum_{p\in M}s_p$.
\item $s_{H\cap \calQ}=q^3+q$ and $\sum_{p\in H\cap \calQ}s_p(s_p-1)=q(q-1)$.
\item If $\ell$ is a line of $H\cap \calQ$, then $s_\ell=q$.
\item Let $U$ be a solid of $H$ such that $\calE:=U\cap \calQ$ is an elliptic quadric $\mathsf{Q}^-(3,q)$ and suppose that $s_p=0$ for at least one point $p\in\calE$. Then $s_\calE=q^2-q$.
\item Let $U$ be a solid of $H$ such that $\calH:=U\cap \calQ$ is a hyperbolic quadric $\mathsf{Q}^+(3,q)$. Then $s_\calH=q^2+q$.
\item Suppose that $q$ is odd. Let $\pi$ be a plane of $H$ such that $C:=\pi\cap \calQ$ is a conic and suppose that $s_x=0$ for at least one point $x$ of $C$. If the line $\ell:=\pi^\perp\cap H$ is external to $\calQ$, then $s_C=q-1$, and if $\ell$ meets $\calQ$ in two points $p$ and $r$, then $s_C=q+1-s_p-s_r$.
\item Suppose that $q$ is even. Let $\pi$ be a plane of $H$ such that $C:=\pi\cap \calQ$ is a conic.

    If $H^\perp\notin \pi$ and $s_p=0$ for at least one point $p$ of $C$, then $\pi^\perp\cap H$ is a line on $H^\perp$ meeting $\calQ$ in exactly one point $u$ and we have $s_C+s_u=q$.

    If $H^\perp\in\pi$, then $\pi^\perp\cap H$ is a plane that meets $\calQ$ in a conic $C'$ and we have $s_C+s_{C'}=2q$.
\end{enumerate}
\end{lemma}
\begin{proof}\ \\
\noindent(a) Let $\ell$ be a line of $\calQ$ that contains $p$ and is not contained in $H$. Let $s$ be the number of points of $Y$ on $\ell$. Let $r$ be a point of $\ell$ with $r$ not equal to $p$. If $r\in Y$, then $r$ is collinear to $q-1$ points of $Y$, of which $s-1$ are on $\ell$, and if $r\notin Y$, then $r$ is collinear to $q^2-q$ points of $Y$, of which $s$ are on $\ell$. Hence there exist exactly $s(q-s)+(q-s)(q^2-q-s)=(q-s)(q^2-q)$ points in $Y$ that are collinear to exactly one of the points $\ne p$ of $\ell$. The remaining points of $Y$ must therefore be collinear to $p$. This shows that $s$ depends only on $p$ but not on the choice of $\ell$.\\

\noindent(b) Each point $p\in H\cap \calQ$ is collinear to $(q^2-q)s_p$ points of $Y$. On the other hand, every point of $Y$ is collinear to $q^2+1$ points of $H\cap \calQ$. A double counting argument thus gives $\sum_{p\in H\cap \calQ}s_p(q^2-q)=|Y|(q^2+1)$. Since $|Y|=q^2(q-1)$, this gives $\sum_{p\in H\cap \calQ}s_p=q^3+q$.

For the second equality in (b) we count triples $(p,r_1,r_2)\in (H\cap \calQ)\times Y\times Y$ of collinear points with $r_1\ne r_2$. Each point $r_1\in Y$ is collinear to $q-1$ points of $Y$ and thus occurs in $q-1$ such triples in the middle position. Each point $p\in H\cap \calQ$ occurs in $(q^2-q)s_p(s_p-1)$ such triples. It follows that
\begin{align*}
\sum_{p\in H\cap \calQ}s_p(s_p-1)(q^2-q)=|Y|(q-1).
\end{align*}
Since $|Y|=q^2(q-1)$, this implies the second equation of statement (b).\\

\noindent(c) A point $p$ of $\ell$ is collinear to $s_p(q^2-q)$ points of $Y$. Since every point of $Y$ is collinear to exactly one point of $\ell$, it follows that $s_\ell(q^2-q)=|Y|$, hence $s_\ell=q$.\\

\noindent(d) The line $U^\perp$ meets $\calQ$ in two points $r$ and $r'$ and these points do not lie in $H$. Since $s_p=0$ for some point of $\calE$, we see that $r,r'\notin Y$. Thus $r$ is collinear to exactly $q^2-q$ points of $Y$ and these are exactly the points of $Y$ on the lines $rz$ with $z\in\calE$. Since $r\notin Y$, it follows that $s_\calE=q^2-q$.\\

\noindent(e) Because a hyperbolic quadric is the union of the $q+1$ lines of any of its two reguli, this follows from part (c).\\

\noindent(f) As $q$ is odd, $\pi^\perp\cap H$ is in fact a line and moreover this line is skew to $\pi$. First consider the case that $\ell$ has no point on the quadric. Then each solid $T$ of $H$ on $\pi$ meets $\calQ$ in an elliptic quadric $\mathsf{Q}^-(3,q)$ or a hyperbolic quadric $\mathsf{Q}^+(3,q)$, and there are $\frac12(q+1)$ solids of each kind. Counting the sum $\sum_{p\in H\cap \calQ}s_p$ using the solids on $\pi$, it follows from (b), (d) and (e) that \begin{align*}
q^3+q+qs_C=\frac12(q+1)(q^2-q)+\frac12(q+1)(q^2+q)
\end{align*}
This implies that $s_C=q-1$.

Now consider the case that the $\pi^\perp\cap H$ is a secant line to the quadric, and let $p$ and $r$ be the two points of $\calQ$ on this line. The solids $\erz{\pi,p}$ and $\erz{\pi,r}$ meet $\calQ$ in a cone $C_p$ resp. $C_r$ with vertex $p$ resp. $r$ over the conic $C$. From the remaining $q-1$ solids of $H$ through $\pi$ one half meets $\calQ$ in an elliptic quadric $\mathsf{Q}^-(3,q)$ and one half meets $\calQ$ in a hyperbolic quadric $\mathsf{Q}^+(3,q)$.
It follows from (c) that $s_{C_p}=q(q+1-s_p)$ and $s_{C_r}=q(q+1-s_r)$. Then a similar counting as in the first case gives
\begin{align*}
q^3+q+qs_C=q(q+1-s_p)+q(q+1-s_r)+\frac12(q-1)(q^2-q)+\frac12(q-1)(q^2+q).
\end{align*}
This implies that $s_C=q+1-s_p-s_r$.\\

\noindent(g) First consider the case when $\pi$ does not contain $H^\perp$. Then $\erz{\pi,H^\perp}$ meets $\calQ$ in a cone $U$ with vertex a point $u$ over the conic $C$. The vertex $u$ is the unique point on the line joining $H^\perp$ to the nucleus of the conic in $\pi$. Part (c) shows that $s_U=q(q+1-s_u)$. Every other solid on $\pi$ meets $\calQ$ in a $\mathsf{Q}^-(3,q)$ or $\mathsf{Q}^+(3,q)$ and there are $q/2$ solids of each type. If $C$ has a point $p$ with $s_p=0$, then the previous parts imply similarly as in the proof of (f) that
\begin{align*}
q^3+q+qs_C&=q(q+1-s_u)+\frac q2(q^2-q)+\frac q2(q^2+q)
\end{align*}
and this gives $s_C+s_u=q$.

Now consider the case when $\pi$ contains $H^\perp$. Then $\tau:=\pi^\perp\cap H$ is a conic plane with $\pi\cap\tau=H^\perp$. Let $C'$ be the conic $\tau\cap\calQ$. The solids on $\pi$ are spanned by $\pi$ and a point of $C'$. Hence all these solids meet $\calQ$ in cones with vertex a point of $C'$. If $u$ is a point of $C'$, then the solid $T:=\erz{\pi,u}$ meets $\calQ$ in the union of $q+1$ lines on $u$, so part (c) shows that $s_T=q(q+1-s_u)$. It follows that
\begin{align*}
q^3+q+qs_C&=\sum_{u\in C'}q(q+1-s_u)=q(q+1)^2-qs_{C'},
\end{align*}
which gives $s_C+s_{C'}=2q$.
\end{proof}

\begin{lemma}\label{isthisused1}
Suppose that $Y$ is an intriguing set of $X$ with $q^2(q-1)$ points such that every point in $Y$ is collinear to $q-1$ points of $Y$, and every point of $X\setminus Y$ is collinear to $q^2-q$ points of $Y$.
Let $\ell$ be a line of $H$ that meets $\calQ$ in exactly two points $p$ and $r$, and let $C$ be the conic $\ell^\perp\cap H\cap \calQ$. Suppose that $s_p=0$. Then $s_C=q(2-s_r)$.
\end{lemma}

\begin{proof}
Put $C=\{z_0,\dots,z_q\}$ and $\pi:=\ell^\perp\cap H$. Then $C=\pi\cap\calQ$.

\noindent\underline{Case 1. We assume that $q$ is odd.}\\
The planes of $H$ on $\ell$ are the planes $\tau$ spanned by $\ell$ and a point $z$ of $\pi$. For $z=z_i\in C$, the plane $\tau:=\erz{\ell,z_i}$ meets $\calQ$ in the union of the lines $z_ip$ and $z_ir$, so we have $s_\tau=2q-s_{z_i}$. Now consider one of the $q^2$ points $z$ of $\pi\setminus C$ and the plane $\tau_z:=\erz{\ell,z}$. Then $\tau_z$ meets $\calQ$ in a conic. The line $\tau_z^\perp\cap H=z^\perp\cap\pi$ is either an exterior line or a secant line of the conic $\pi\cap\calQ$, depending on whether or not $z$ is an interior or exterior point of the conic $C$. Thus, exactly ${q\choose 2}$ planes $\tau_z$ have the property that $\tau_z^\perp\cap H$ is an exterior line of the quadric and these planes satisfy $s_{\tau_z}=q-1$ by Lemma \ref{properties} (f). The remaining ${q+1\choose 2}$ planes $\tau_z$ of $H$ on $\ell$ can be indexed $\tau_{ij}$ with $i<j$ where $z_i$ and $z_j$ are the two points of $C$ that lie on the line $\erz{z,\ell}^\perp\cap\pi$. Then $s_{\tau_{ij}}=q+1-s_{z_i}-s_{z_j}$ again by Lemma \ref{properties} (f). Counting the sum $\sum_{p\in H\cap \calQ}s_p$ by considering the planes of $H$ on $\ell$ we thus find
\begin{align*}
q^3+q+(q^2+q)s_\ell&=\sum_{i=0}^q(2q-s_{z_i})+{q\choose 2}(q-1)+\sum_{i,j=0\atop i<j}^q(q+1-s_{z_i}-s_{z_j}).
\end{align*}
For each $i$ the term $s_{z_i}$ occurs $q+1$ times in this equation. It follows thus that
\begin{align*}
q^3+q+(q^2+q)s_\ell&=(q+1)\cdot 2q+{q\choose 2}(q-1)+{q+1\choose 2}(q+1)-(q+1)\sum_{i=0}^qs_{z_i}.
\end{align*}
As $s_C=\sum_is_{z_i}$ this gives $qs_\ell=2q-s_C$.
By Lemma \ref{properties}(f), we have
\[
q s_C=q(q+1-s_\ell)=q^2+q-(2q-s_C)=q^2-q+s_C
\]
and hence $s_C=q$ and $s_\ell=1$. Therefore, $s_r=s_p+s_r=s_\ell=1$ and $s_C=q(2-s_r)$.

\medskip

\noindent\underline{Case 2. We assume that $q$ is even.}\\
As before the planes of $H$ on $\ell$ are the planes spanned by $\ell$ and a point of $\pi$. Let $z$ be a point of $\pi$ and $\tau=\erz{\ell,z}$. If $z=z_i\in C$, then  $s_\tau=2q-s_{z_i}$ as in Case 1. If $z=H^\perp$, then $\tau^\perp\cap H=\pi$, so $s_\tau=2q-s_\pi$ by Lemma \ref{properties} (g). If $z$ is one of the $q^2-1$ points $z$ with $z\ne H^\perp$ and $z\notin C$, then the line on $H^\perp$ and $z$ meets $C$ in a unique point $z_i$ and Lemma \ref{properties} (g) shows $s_\tau=q-s_{z_i}$. Notice that every point $z_i$ of $C$ occurs for exactly $q-1$ choices for $z$ of $\pi$ with $z\ne H^\perp$ and $z\notin C$. Counting the sum $\sum_{p\in H\cap \calQ}s_p$ by considering the planes of $H$ on $\ell$ we thus find
\begin{align*}
q^3+q+(q^2+q)s_\ell&=\sum_i(2q-s_{z_i})+(2q-s_C)+(q-1)\sum_i(q-s_{z_i}).
\end{align*}
Using $s_C=\sum_is_{z_i}$, this simplifies to $(q^2+q)s_\ell=2q(q+1)-(q+1)s_{C}$.
Since $s_\ell=s_p+s_r=s_r$, this proves the statement.
\end{proof}

\begin{lemma}\label{isthisused2}
Suppose that $q$ is odd and that $Y$ is an intriguing set with $q^2(q-1)$ points such that every point in $Y$ is collinear to $q-1$ points of $Y$, and every point of $X\setminus Y$ is collinear to $q^2-q$ points of $Y$. Then $s_r\in\{0,1,q\}$ for every point $r\in H\cap\calQ$.
\end{lemma}
\begin{proof}
 Since $\sum_{w\in H\cap \calQ}s_w(s_w-1)=q(q-1)$, we have $s_r\le q$. We may assume that $0<s_r<q$. Let $\ell$ be a line of $H\cap\calQ$ on $r$. It follows from Lemma \ref{properties} (c) that $\ell$ contains a point $z$ with $z\ne r$ and $s_z>0$. Let $h$ be a line of $H\cap \calQ$ on $z$ with $h\ne \ell$. Lemma \ref{properties} (c) shows that $h$ contains a point $p$ with $s_p=0$. Then $pr$ is a secant line
 and $z$ is a point of the conic $C:=(pr)^\perp\cap H\cap \calQ$. Hence $s_C\ge s_z>0$. Lemma \ref{isthisused1} shows that $s_C=q(2-s_r)$. Hence $s_r<2$, that is $s_r=1$.
\end{proof}

\begin{theorem}\label{smalltype4}
Suppose $Y$ is an intriguing set of the Pentilla-Williford scheme $(X,\{R_0,\dots,R_4\})$ with $|Y|=q^2(q-1)$ such that every point in $Y$ is collinear to $q-1$ points of $Y$, and every point of $X\setminus Y$ is collinear to $q^2-q$ points of $Y$. Then there exists a point $p\in H\cap \calQ$ such that $Y$ consists of the $(q^2-q)q$ points of $X=\calQ\setminus H$ that are collinear to $p$.
\end{theorem}
\begin{proof}
Since $s_p\in\{0,1,q\}$ for all points $p\in H\cap \calQ$, Lemma \ref{properties} (b) implies that there exists a unique point $p$ with $s_p=q$. Then each of the $q^2-q$ lines of $\calQ$ on $p$ that is not contained in $H$ meets $Y$ in $q$ points. Hence, $Y$ consists of the points of $X$ that are collinear to $p$.
\end{proof}

%\bibliographystyle{abbrv}
%\bibliography{citations}

\end{document}